\documentclass[11pt, reqno]{amsart}

\usepackage{enumitem}

\usepackage{geometry}
\usepackage{color}
\usepackage{amssymb}
\usepackage{amsmath}
\usepackage{arydshln}
\usepackage{hyperref}
\usepackage{bbm}
\usepackage{mathrsfs}

\def\BibTeX{{\rm B\kern-.05em{\sc i\kern-.025em b}\kern-.08em
    T\kern-.1667em\lower.7ex\hbox{E}\kern-.125emX}}

\numberwithin{equation}{section}

\newcommand{\R}{\mathbb{R}}

\newcommand{\N}{\mathbb{N}}
\newcommand{\Z}{\mathbb{Z}}

\newcommand{\B}{\mathbb{B}}

\newcommand{\E}{\mathbb{E}}
\newcommand{\e}{\varepsilon}

\newcommand{\1}{\mathbbm{1}}

\newtheorem{Theorem}{Theorem}[section]
\newtheorem{Proposition}[Theorem]{Proposition}
\newtheorem{Corollary}[Theorem]{Corollary}
\newtheorem{Lemma}[Theorem]{Lemma}
\newtheorem{Remark}[Theorem]{Remark}
\newtheorem{Definition}[Theorem]{Definition}

\begin{document}

\title[Long-time behavior for subcritical measure-valued branching processes]{Long-time behavior for subcritical measure-valued branching processes with immigration}

\author{Martin Friesen}
\address[Martin Friesen]{Fakult\"at f\"ur Mathematik und Naturwissenschaften\\ Bergische Universit\"at Wuppertal\\ Gau\ss stra\ss e 20, 42119 Wuppertal, Germany}
\email{friesen@math.uni-wuppertal.de}

\date{\today}



\begin{abstract}
In this work we study the long-time behavior for subcritical
measure-valued branching processes with immigration
on the space of tempered measures.
Under some reasonable assumptions on the spatial motion, the branching and immigration mechanisms, we prove the existence and uniqueness of an invariant measure for the corresponding Markov transition semigroup. 
Moreover, we show that it converges with exponential rate to the unique invariant measure in the Wasserstein distance as well as in a distance defined in terms of Laplace transforms. Finally, we consider an application of our results to super-L\'evy processes as well as branching particle systems on the lattice with noncompact spins.
\end{abstract}

\maketitle

\allowdisplaybreaks

\section{Introduction}
Measures-valued branching processes have been first studied by
Watanabe 1968 \cite{W68}, Silverstein 1969 \cite{S69} and Kawazu, Watanabe 1971 \cite{KW71}
where they have been derived as scaling limits of Galton-Watson processes.
For a detailed introduction on measure-valued Markov processes (also called superprocesses) we refer to Dynkin \cite{D94}, Etheridge \cite{E00}, Perkins \cite{P02}, Le Gall \cite{LG99} and Li \cite{L11}.
A measure-valued branching process with immigration is a Markov process whose Markov transition kernel $P_t(\mu,d\nu)$ has Laplace transform
\begin{align*}
 \int \limits_{M(E)} e^{- \langle f, \nu \rangle} P_t(\mu,d\nu) 
 = \exp\left( - \langle V_tf, \mu \rangle - \int \limits_0^t \psi(V_sf)ds \right), \qquad t \geq 0,
\end{align*}
where $\langle f,\nu \rangle = \int_E f(x) \nu(dx)$ and $v_t(x) := V_tf(x)$ is the unique nonegative mild solution to 
\begin{align}\label{EQ:11}
 \frac{\partial v_t(x)}{\partial t} = Av_t(x) - \phi(x, v_t), \qquad v_0 = f \in D(A).
\end{align}
Here $E$ denotes the location space, $(A,D(A))$ the generator of a transition semigroup describing the underlying spatial motion, $M(E)$ the space of all finite Borel measures on $E$, $\phi$ the branching and $\psi$ the immigration mechanism.

One of the most prominent examples is the super-Brownian motion
where $E = \R^d$, $A = \frac{1}{2}\Delta$, $\psi = 0$ and
$\phi(x,f) = \frac{1}{2}f(x)^2$, see, e.g., \cite{KS88}, \cite{R89} and \cite{S90} where it is shown that in dimension $d = 1$ this process has a density field $(X_t(x))_{x, \in \R, \ t \geq 0}$ which is the unique weak solution to 
\begin{align}\label{EQ:35}
 \frac{\partial X_t(x)}{\partial t} = \frac{1}{2}\Delta X_t(x) + \sqrt{X_t(x)}\overset{\textbf{.}}{W}_t(x), \qquad t > 0, \ \ x \in \R,
\end{align}
where $\overset{\textbf{.}}{W}_t(x)$ is the derivative of a space-time Gaussian white noise. Note that pathwise uniqueness for \eqref{EQ:35} remains an open problem. However, still in dimension $d = 1$, pathwise uniqueness for the distribution function was recently obtained in \cite{X13}.
The latter result was extended in \cite{HLY14} to super-L\'evy processes with branching mechanism
\begin{align}\label{EQ:23}
 \phi(x,f) = bf(x) + cf(x)^2 + \int \limits_0^{\infty}\left( e^{- z f(x)} - 1 + z f(x) \right) m(dz),
\end{align}
where $b \in \R$, $c \geq 0$ and $m$ is a Borel measure on $(0,\infty)$ with $\int_0^{\infty}u \wedge u^2 m (du) < \infty$. Let us also mention other interesting related results such as \cite{C15}, \cite{DL12} and \cite{MX15}.
It is worthwile to mention that solutions to nonlinear Partial Differential Equations of the form \eqref{EQ:11} can be simulated by a Monte Carlo algorithm based on the associated branching process, see \cite{HTT14} and the references therein.

Another natural and interesting example studied in the literature corresponds to branching particle systems where $E$ is a countable discrete set and $A$ the generator of a Markov chain on $E$. Indeed, the particular case $E = \N$ with $A = 0$
was recently studied by Kyprianou and Palau \cite{KP18} where precise conditions for local and global extinction have been obtained. 
The case $E = \Z^d$ falls into the framework of interacting particle systems on the lattice with noncompact spins as studied by Bezborodov, Kondratiev and Kutoviy \cite{BKK15}, see also the references therein. Note that in both cases we may identify 
\[
 M(E) \cong \left \{ (\mu(x))_{x \in E} \in \R_+^E \ | \ \sum \limits_{x \in E}\mu(x) < \infty \right\},
\]
which reveals the underlying lattice structure.

In this work we study the long-time behavior of subcritical measure-valued branching processes including the aforementioned cases of super-L\'evy processes as well as branching particle systems on the lattice with noncompact spins.
We provide a sufficient condition on the branching and immigration mechanisms such that the associated Markov process admits a unique invariant measure $\pi$. Afterwards we establish $P_t(\mu,\cdot) \longrightarrow \pi$ for $t\to \infty$
and study the corresponding rate of convergence
in two different distances.

Let us briefly mention some known results from the literature.
In the particular case $E = \{0\}$ one has $M(\{0\}) \cong \R_+$ and the most general branching mechanism is of the form \eqref{EQ:23}. It describes a one-dimensional \textit{continuous-state branching process with immigration} first introduced in the framework of diffusions by Feller 1951 \cite{F51} and then for more general cases by Ji\v{r}ina 1958 \cite{J58}. Its long-time behavior was studied in \cite{KM12}, \cite[Chapter 3]{L11}, \cite{LM15} and more recently in \cite{FJKR19}. More generally, the case $E = \{1,\dots, d\}$ with $A = 0$ corresponds to \textit{multi-type continuous-state branching processes with immigration} for which the corresponding state space is identified by $M(\{1,\dots, d\}) \cong \R_+^d$, see \cite{BLP15}. In contrast to $E = \{0\}$, here the most general branching mechanism involves also the coupling of different components in $\R_+^d$ which results in a nonlocal branching mechanism.
The first general result on invariant measures and ergodicity was recently obtained in \cite{JKR18} and \cite{FJR18}, 
see also \cite{2018arXiv181110542M} and \cite{2018arXiv181100122Z}
for some related results on this topic.
For arbitrary compact spaces $E$, Stannat has studied in \cite{S03b} and \cite{S03a} the case of an immigration mechanism $\psi(f) = \langle f, \beta \rangle$ and branching mechanism
\begin{align}\label{EQ:19}
 \phi(x,f) &= c(x)f(x)^2 + b(x)f(x) + \int \limits_{0}^{\infty}\left( e^{- f(x)z} -1 + f(x)z \right) m(x,dz),
\end{align}
where $\beta$ is a finite measure on $E$,
$c \geq 0$ and $b$ are continuous and bounded while $m(x,dz)$
is a kernel of positive measures satisfying $\sup_{x \in E}\int_0^{\infty} z \wedge z^2 m(x,dz)< \infty$. 
Note that $\phi(x,f)$ given by \eqref{EQ:19} is local in $f$ and describes at each point $x \in E$ the branching mechanism of a \textit{one-dimensional continuous-state branching process}. 

Contrary, in this work $E$ is not supposed to be compact and $\phi$, $\psi$ are general (including nonlocal) branching and immigration mechanisms in the sense of \cite[Chapters 2, 6 and 9]{L11}, see Section 2 here for details. In particular, our results also cover 
the case of super-L\'evy processes with immigration as well as branching particle systems on the lattice as studied in \cite{KP18}. 
Contrary to the classical case of measure-valued branching processes supported on the space of finite Borel measures $M(E)$, the results obtained in this work also cover branching Markov processes supported on the space of tempered measures over $E$. 
The latter ones can be used to treat population models where the total number of individuals is infinite. Tempered measures have been first introduced for the study of critical branching Markov processes without immigration $(\psi = 0)$ where the long-time behavior of the Markov process started  at the Lebesgue measure was investigated, see Iscoe \cite{I86}.
Further developments in this direction have been obtained in \cite{BCG93,BCG97, MR0478374, E93, KKP08} and \cite{KPZ16}. 

Another motivation to study branching processes on the space of tempered measures is related to possible immigration of mass (or particles). Namely, consider a branching process on $E = \R^d$ and suppose that mass immigrates according to a Poisson point process with intensity $\beta(dx)$. 
If $\beta$ is a finite measure (e.g. $\beta(dx) = \1_{ \{ |x| \leq 1\}}dx$ or $\beta(dx) = \delta_0(dx)$), then only a finite amount of mass immigrates into the system in finite time. However, applications modelled by translation invariant rates such as $\beta(dx) = dx$ share the common feature that 
that $\beta$ is an infinite measure and hence an infinite amount of mass enters the system in finite time. The temperedness of measures (see Section 2) is used to prevent the system to accumulate an infinite amount of mass in some compact set.

This work is organized as follows.
In Section 2 we introduce the class of measure-valued branching processes with immigration studied in this work. Our main reference for this section is \cite{L11} where their construction and properties are studied. 
Although all results stated in Section 2 should be well-known among experts, we provide, whenever we were not able to find an adequate reference, some additional comments and sketches of proofs. In this way we intend to make this work accessible to a wide audience.
Section 3 is devoted to the existence and uniqueness of invariant measures for the processes introduced in Section 2 whereas an exponential rate of convergence towards the unique invariant measure is investigated in Section 4. Finally the example of super-L\'evy processes and measure-valued particle systems is considered in Section 5. Some auxilliary results are collected in the appendix.

\section{Measure-valued branching processes with immigration}

\subsection{Preliminaries}
Let $E$ be a Lusin topological space (e.g. a Polish space).
Fix a strictly positive, continuous function $h$ on $E$.
Let $B_h(E)$ be the Banach space of real-valued Borel functions on $E$ equipped with norm
\[
 \| f \|_h := \sup \limits_{x \in E}\frac{|f(x)|}{h(x)} < \infty.
\]
Denote by $C_h(E)$ its closed subspace of continuous functions 
satisfying $\|f\|_{h} < \infty$ and let $B_h(E)^+, C_h(E)^+$ be the corresponding cones of nonnegative functions.
Let $M_h(E)$ be the space of tempered measures over $E$, i.e.
\[
 M_h(E) = \left\{ \mu \text{ Borel measure on } E \ : \  \int \limits_{E}h(x)\mu(dx) < \infty \right\}.
\]
For $f \in B_h(E)$ and $\mu \in M_h(E)$ we set
\[
 \langle f, \mu \rangle = \int \limits_{E}f(x) \mu(dx).
\]
A topology on $M_h(E)$ can be defined by the convention 
\begin{align}\label{EQ:21}
 \mu_n \longrightarrow \mu \text{ in } M_h(E) \qquad \Longleftrightarrow \qquad
 \langle f, \mu_n \rangle \longrightarrow \langle \mu, f \rangle, \ \ \forall f \in C_h(E).
\end{align}
Denote by $\mathcal{B}(M_h(E))$ the corresponding Borel-$\sigma$-algebra. Let $\1$ stand for the constant function equal to $1$. Then $B_{\1}(E)$ is the space of bounded Borel functions while $M_{\1}(E)$ denotes the space of finite Borel measures over $E$. Using the homeomorphisms $M_h(E) \ni \mu \longmapsto h(x)\mu(dx) \in M_{\1}(E)$ and $B_{\1}(E) \ni f \longmapsto hf \in B_h(E)$ combined with 
\cite[Corollary 1.12]{L11} one can show that
\begin{align}\label{EQ:22}
 \mathcal{B}(M_h(E)) = \sigma( \{ \mu \longmapsto \langle f, \mu \rangle \ | \ f \in B_h(E)^+ \} ).
\end{align}
Write $M_h(E)^{\circ} = M_h(E) \backslash \{ 0\}$ where $0$ denotes the zero measure.
We endow $M_h(E)^{\circ}$ with the restriction of the topology from $M_h(E)$ and the corresponding Borel-$\sigma$-algebra.

Let $\mathcal{P}(M_h(E))$ be the space of all Borel probability measures over $M_h(E)$. 
For $\rho \in \mathcal{P}(M_h(E))$ the Laplace transform of $\rho$ is defined by
\[
 \mathcal{L}_{\rho}(f) = \int \limits_{M_h(E)} e^{- \langle f, \nu \rangle} \rho(d\nu), \qquad f \in B_h(E)^+.
\]
We say that $(f_n)_{n \geq 1} \subset B_h(E)$ converges bounded pointwise to some $f \in B_h(E)$, if $\sup_{n \geq 1}\| f_n \|_{h} < \infty$ and
$f_n \longrightarrow f$ pointwise as $n \to \infty$.
In this work we will frequently use the following classical properties of the Laplace transform.
\begin{Proposition}\label{LEMMA:01}
 The following assertions hold:
 \begin{enumerate}[leftmargin = *]
  \item[(a)] Given $\rho, \widetilde{\rho} \in \mathcal{P}(M_h(E))$ such that 
  $\mathcal{L}_{\rho}(f) = \mathcal{L}_{\widetilde{\rho}}(f)$ for all $f \in C_h(E)^+$, then $\rho = \widetilde{\rho}$.
  \item[(b)] Let $(\rho_n)_{n \geq 1} \subset \mathcal{P}(M_h(E))$ and $\rho \in \mathcal{P}(M_h(E))$. 
  Then $\rho_n \longrightarrow \rho$ weakly in $\mathcal{P}(M_h(E))$ if and only if $\lim_{n \to \infty} \mathcal{L}_{\rho_n}(f) = \mathcal{L}_{\rho}(f)$ for all $f \in C_h(E)^+$.
  \item[(c)] Let $\mathcal{L}$ be a functional on $B_h(E)^+$
  continuous at $f = 0$ with respect to bounded pointwise convergence. Suppose that there exists $(\rho_n)_{n \geq 1} \subset \mathcal{P}(M_h(E))$ 
  with $\lim_{n \to \infty}\mathcal{L}_{\rho_n}(f) = \mathcal{L}(f)$ for all $f \in B_h(E)^+$.
  Then there exists $\rho \in \mathcal{P}(M_h(E))$ such that $\mathcal{L}_{\rho} = \mathcal{L}$.
 \end{enumerate}
\end{Proposition}
\begin{proof}
 Using the homeomorphisms $M_h(E) \ni \mu(dx) \longmapsto h(x)\mu(dx) \in M_{\1}(E)$ and $B_{\1}(E) \ni f \longmapsto h f \in B_{h}(E)$ the assertions are particular cases of the properties of Laplace transforms on $M_{\1}(E)$, see e.g. \cite[Chapter 1]{L11}.
\end{proof}

\subsection{Construction}
Below we describe the class of measure-valued branching processes with immigration studied in this work. 
Let $E$ be a Lusin topological space and $h$ a strictly positive continuous function on $E$. 
The spatial motion of is described by a process
satisfying the following condition:
\begin{enumerate}[leftmargin = *]
 \item[(A1)] Let $p^{\xi}_t(x,dy)$ be the transition kernel of a conservative Borel right process $\xi = (\xi_t)_{t \geq 0}$ on $E$. Suppose that there exists $\alpha \geq 0$ such that
\begin{align}\label{EQ:05}
 \lim \limits_{t \to 0}e^{- \alpha t}\int \limits_{E}h(y) p_t^{\xi}(x,dy) = h(x)
\end{align}
increasingly for each $x \in E$.
\end{enumerate}
Note that each Feller process is a Borel right process. 
Property \eqref{EQ:05} states that $h$ is an $\alpha$-excessive function for the transition semigroup. 
The following is a useful sufficient condition for \eqref{EQ:05}.
\begin{Remark}
 Let $(A,D(A))$ be the extended generator of $\xi$. If $h \in D(A)$ and there exists $\alpha \geq 0$ such that $Ah \leq \alpha h$, then \eqref{EQ:05} is satisfied.
\end{Remark}
For $x \in E$ and $f \in B_h(E)^+$ introduce the branching mechanism
 \begin{align*}
 \phi(x,f) &= c(x)f(x)^2 + b(x)f(x) - \int \limits_{E}f(y)\eta(x,dy)
 \\ &\ \ \ + \int \limits_{M_h(E)^{\circ}}\left( e^{- \langle f,\nu\rangle} - 1 + f(x) \nu(\{x\}) \right) H_1(x,d\nu)
\end{align*}
and immigration mechanism
 \begin{align}\label{EQ:13}
  \psi(f) = \langle f, \beta \rangle + \int \limits_{M_h(E)^{\circ}}\left( 1 - e^{- \langle f, \nu \rangle}\right)H_2(d\nu).
\end{align}
The corresponding parameters are supposed to satisfy the following conditions:
\begin{enumerate}[leftmargin = *]
 \item[(A2)] $b \in B_{\1}(E)$, $ch \in B_{\1}(E)^+$, $\eta$ a $\sigma$-finite kernel on $E$ and $H_1(x, d\eta)$ a $\sigma$-finite kernel from $E$ to $M_h(E)^{\circ}$ satisfying, for each $x \in E$,
 \[
  \int \limits_{E}h(y) \eta(x,dy) 
  + \int \limits_{M_h(E)^{\circ}}\left( \langle h,\nu \rangle \wedge \langle h, \nu \rangle^2 + \langle h, \nu_x \rangle \right) H_1(x,d\nu) \leq C h(x),
 \]
 for some constant $C > 0$, 
 where $\nu_x$ denotes the restriction of $\nu(dy)$ to $E \backslash \{x\}$.
 \item[(A3)] $\beta \in M_h(E)$ and $H_2$ is a Borel measure on $M_h(E)^{\circ}$ satisfying
 \begin{align}\label{EQ:09}
  \int \limits_{M_h(E)^{\circ}} 1 \wedge \langle h, \nu \rangle H_2(d\nu) < \infty.
 \end{align}
\end{enumerate}
The next result provides the construction of the measure-valued branching processes with immigration.
\begin{Theorem}
 Suppose that conditions (A1) -- (A3) are satisfied. Then
 \begin{enumerate}[leftmargin = *]
  \item[(a)] For each $f \in B_h(E)^+$ there exists a unique nonnegative solution $\R_+ \times E \ni (t,x) \longmapsto v_t(x,f)$ to 
  \begin{align}\label{EQ:00}
   v_t(x,f) = \int \limits_{E}f(y)p_t^{\xi}(x,dy)
   - \int \limits_{0}^{t} \int \limits_{E} \phi(y,v_s(\cdot,f)) p^{\xi}_{t-s}(x,dy) ds,
  \end{align}
  so that $t \longmapsto \| v_t(\cdot,f)\|_h$ is bounded on each bounded interval $[0,T]$.
  \item[(b)] Letting $V_tf(x) = v_t(x,f)$, we find that $(V_t)_{t \geq 0}$ is a comulant semigroup in the sense that $V_tV_s = V_{t+s}$ for all $t,s \geq 0$ and
  \begin{align}\label{EQ:30}
   V_tf(x) = \int \limits_{E}f(y) \lambda_t(x,dy) 
   + \int \limits_{M_h(E)^{\circ}}\left( 1 - e^{- \langle f, \nu \rangle} \right) L_t(x,d\nu),
  \end{align}
  where $\lambda_t(x,dy)$ and $L_t(x,d\nu)$ are transition kernels satisfying
  \[
   \int \limits_{E}h(y)\lambda_t(x,dy) + \int \limits_{M_h(E)^{\circ}} 1 \wedge \langle h, \nu \rangle L_t(x,d\nu) \leq C_t h(x), \qquad x \in E, \ \ t \geq 0,
  \]
  for some constant $C_t > 0$.
  \item[(c)] There exists a unique Markov kernel $P_t(\mu,d\nu)$ whose Laplace transform is, for $t \geq 0$, $\mu \in M_h(E)$ and $f \in B_h(E)^+$, given by
 \begin{align}\label{BRANCHING:LAPLACE}
  \int \limits_{M_h(E)}e^{- \langle f, \nu\rangle} P_t(\mu,d\nu)
  = \exp\left( - \langle V_tf, \mu \rangle - \int\limits_{0}^{t}\psi(V_sf)ds \right).
 \end{align}
 \end{enumerate}
\end{Theorem}
This result can be obtained from \cite{L11}, details are postponed to the appendix.
\begin{Remark}
 Let $(A,D(A))$ be the extended generator of $\xi$. Then \eqref{EQ:00} is a mild formulation of \eqref{EQ:11}.
\end{Remark}
In this work we will always assume that conditions (A1) -- (A3) are satisfied and call the corresponding Markov process 
$(\xi, \phi, \psi)$-superprocess. Under some additional 
conditions one may also identify the generator of this process in terms of a martingale problem. 
The latter one is for some class of functions $F: M_h(E) \longrightarrow \R$ formaly given by
\begin{align*}
 \mathcal{A}F(\mu) &= \int \limits_{E}\mu(dx)\left( AF'(\mu;x) - b(x)F'(\mu;x) + \int \limits_{E}F'(\mu;y)\eta(x,dy) + c(x)F''(\mu;x) \right)
 \\ &\ \ \ + \int \limits_{E}\mu(dx) \int \limits_{M_h(E)^{\circ}} \left( F(\mu + \nu) - F(\mu) - F'(\mu;x)\nu(\{x\}) \right)H_1(x,d\nu)
 \\ &\ \ \ + \int \limits_{E}F'(\mu;x)\beta(dx) 
 + \int \limits_{M_h(E)^{\circ}}\left( F(\mu + \nu) - F(\mu) \right) H_2(d\nu)
\end{align*}
 where the variational derivaties are defined by
\[
 F'(\mu;x) = \lim \limits_{\e \to 0} \frac{F(\mu + \e \delta_x) - F(\mu)}{\e}, \ \ F''(\mu;x) = \lim \limits_{\e \to 0} \frac{F'(\mu + \e \delta_x;x) - F'(\mu;x)}{\e}.
\]
Since we will not make use of the corresponding martingale problem characterization, we refer the reader to 
\cite[Section 9]{L11} for additional details.

\subsection{Behavior of first moment}
In order to study the first moment of $P_t(\mu,d\nu)$ it is reasonable to rewrite the branching mechanism to
\begin{align*}
 \phi(x,f) &= c(x)f(x)^2 + b(x)f(x) - \int \limits_{E}f(y)\gamma(x,dy)
 \\ &\ \ \ + \int \limits_{M_h(E)^{\circ}}\left( e^{- \langle f,\nu\rangle} - 1 + \langle f, \nu \rangle \right) H_1(x,d\nu),
\end{align*}
where $\gamma(x,dy)$ is defined by
\[
 \gamma(x,dy) = \eta(x,dy) + \int \limits_{M_h(E)^{\circ}}\nu_x(dy)H_1(x,d\nu).
\]
In this way the contribution of the branching to the first moment is encoded in the properties of the kernel $\gamma$ and its associated transition operator 
\[
 \Gamma f(x) = \int \limits_{E}f(y) \gamma(x,dy), \qquad f \in B_h(E).
\]
Denote by $(R_t)_{t \geq 0}$ the unique semigroup on $B_h(E)$ obtained from
 \begin{align}\label{EQ:34}
  R_tf(x) = p_t^{\xi}f(x) + \int \limits_{0}^{t} \int \limits_{E} \left(\Gamma R_{t-s}f(y) - b(y)R_{t-s}f(y) \right)p_s^{\xi}(x,dy)ds.
 \end{align}
 Indeed, as $\Gamma$ and the multiplication operator $(-bf)(x) = -b(x)f(x)$ are both bounded on $B_h(E)$, existence of a solution can be classically obtained by iterating \eqref{EQ:34}. Uniqueness is a straightforward consequence of the Gronwall lemma. 
 Let us stress that $(R_t)_{t \geq 0}$ and $(p_t^{\xi})_{t \geq 0}$ need not to be strongly continous neither on $B_h(E)$ nor on $C_h(E)$. Moreover, $(R_t)_{t \geq 0}$ does not need to be conservative, i.e. $R_t 1 \neq 1$ is allowed. We will need the following simple properties of this semigroup.
\begin{Lemma}\label{LEMMA:06}
 This semigroup satisfies for each $t \geq 0$, $x \in E$ and $f \in B_h(E)^+$
 \begin{align}\label{EQ:33}
  R_tf(x) = \int \limits_{E}f(y)\lambda_t(x,dy) 
  + \int \limits_{M_h(E)^{\circ}}\langle f, \nu \rangle L_t(x,d\nu)
 \end{align}
 and 
 \[
  \frac{1}{\e}V_t(\e f)(x) \nearrow R_tf(x), \qquad \e \searrow 0.
 \]
\end{Lemma}
The proof of this lemma is postponed to the appendix. Set 
\[
 r_t(x,dy) = \lambda_t(x,dy) + \int \limits_{M_h(E)^{\circ}}\nu(dy) L_t(x,dy).
\]
Using \eqref{EQ:33} one can show that for each $f \in B_h(E)$ and $t \geq 0$ 
\begin{align}\label{EQ:15}
 R_tf(x) = \int \limits_{E}f(y)r_t(x,dy), \qquad x \in E,
\end{align}
The adjoint action of $(R_t)_{t \geq 0}$ on $M_h(E)$ is then given by
\[
 R_t^*\nu(dy) = \int \limits_{E} r_t(x,dy)\nu(dx), \qquad t \geq 0.
\]
This semigroups describe the evolution of the first moment for the $(\xi, \phi,\psi)$-superprocess.
\begin{Proposition}\label{LEMMA:00}
 Suppose that $H_2$ satisfies 
 \begin{align}\label{EQ:06}
    \int \limits_{ \{ \nu \in M_h(E)^{\circ} \ : \ \langle h, \nu \rangle > 1 \} } \langle h, \nu \rangle H_2(d\nu) < \infty.
 \end{align}
 Then for each $\mu \in M_h(E)$ one has
  \[
   \int \limits_{M_h(E)} \nu(dx) P_t(\mu,d\nu)
   = R_t^*\mu(dx) + \int \limits_{0}^{t}R_s^*a(dx)ds,
  \]
  where both integrals are understood in a weak sense in $M_h(E)$, and
  \[
   a(dx) = \beta(dx) + \int \limits_{M_h(E)^{\circ}}\nu(dx) H_2(d\nu).
  \]
\end{Proposition}
\begin{proof}
 If $h \equiv 1$, then the assertion follows from \cite[Proposition 9.11]{L11}. A similar result can be also deduced in the case $h \neq 1$.
 Since we could not find a precise reference, 
 for convenience of the reader a proof is given in the appendix.
\end{proof}
This result suggests to relate the long-time behavior of $(\xi, \phi, \psi)$-superprocesses with the long-time behavior of the semigroup $(R_t)_{t \geq 0}$.
\begin{Definition}
 The $(\xi, \psi, \phi)$-superprocess is called subcritical, 
 if there exists constants $C, \delta > 0$ such that
 \begin{align}\label{EQ:07}
  \int \limits_{E}h(y)r_t(x,dy) \leq C h(x)e^{- \delta t}, \qquad t \geq 0, \ \ x \in E.
 \end{align}
\end{Definition}
\begin{Remark}
 Property \eqref{EQ:07} implies that, for each $f \in B_h(E)^+$, 
 \[
  R_tf(x) \leq \| f\|_h \int \limits_{E}h(y) r_t(x,dy) 
  \leq \| f\|_h Ch(x) e^{-\delta t}, \qquad t \geq 0, \ \ x \in E.
 \]
 In the language of operator semigroups such a property states that $(R_t)_{t \geq 0}$ is uniformly exponentially stable.
 However, in contrast to the classical theory of operator semigroups we do not require that $(R_t)_{t \geq 0}$ is strongly continuous.
\end{Remark}
Such condition implies that mass located inside the system exponentially decreases in time. However, as additional mass also immigrates into the system, it is reasonable to expect the existence of at least one invariant measure. This is precisely the content of Section 3.
Let us first start with the analysis of the first moment.
\begin{Corollary}
 Let $P_t(\mu,d\nu)$ be the Markov kernel of a subcritical $(\xi,\phi,\psi)$-superprocess satisfying \eqref{EQ:06}. 
 Then for each $\mu \in M_h(E)$ one has
 \[
  \lim \limits_{t \to \infty}\int \limits_{M_h(E)} \nu(dx) P_t(\mu,d\nu) = \int \limits_{0}^{\infty} R_s^*a(dx) ds.
 \]
 in the weak topology on $M_h(E)$.
\end{Corollary}
\begin{proof}
 Take $f \in B_h(E)$, then
 $\langle f, R_t\mu \rangle = \langle R_t f, \mu \rangle 
  \leq C \| f \|_h e^{- \delta t} \langle h,\mu \rangle$
 and analogously we obtain
 \begin{align*}
  \langle f, R_s^*a \rangle \leq C \| f\|_h \langle h, \beta \rangle e^{- \delta s} + C \|f \|_h e^{-\delta s}\int \limits_{M_h(E)^{\circ}} \langle h, \nu \rangle H_2(d\nu).
 \end{align*}
 This readily implies the assertion.
\end{proof}
Let us close this section with a Lyapunov-type condition 
implying \eqref{EQ:07}.
\begin{Proposition}
 Suppose that there exists $\delta > 0$ such that
  \begin{align}\label{EQ:20}
   Ah(x) + \Gamma h(x) \leq b(x)h(x) - \delta h(x), \qquad x \in E,
  \end{align}
  and assume that $h$ satisfies
  \[
   R_th(x) = h(x) + \int \limits_{0}^{t} R_s\left( Ah - bh + \Gamma h \right)(x) ds, \qquad x \in E, \ \ t \geq 0,
  \]
  where it is implicitly assumed that the integral on the right-hand side exists.
  Then \eqref{EQ:07} holds with $C = 1$.
\end{Proposition}
\begin{proof}
 Using the product rule and then \eqref{EQ:20} we find that
 \[
  e^{\delta t}R_th(x) = h(x) + \int \limits_{0}^{t}e^{\delta s} R_s\left( Ah - bh + \Gamma h + \delta h \right) ds
  \leq h(x),
 \]
 where we have used that $R_s$ is a positive operator. 
 Hence $R_th(x) \leq e^{- \delta t}h(x)$ which proves the assertion.
\end{proof}

\section{Invariant measure and stationarity}
In this section we provide a sufficient condition for the existence and uniqueness of an invariant measure for a subcritical $(\xi,\phi,\psi)$-superprocess with Markov kernel $P_t(\mu,d\nu)$.
Recall that $\pi \in \mathcal{P}(M_h(E))$ is an invariant measure, if 
\begin{align}\label{EQ:17}
 \int \limits_{E} P_t(\mu, d\nu) \pi(d\mu) = \pi(d\nu), \qquad t \geq 0.
\end{align}
The following is our main result of this section.
\begin{Theorem}\label{TH:00}
 Suppose that the $(\xi,\phi,\psi)$-superprocess is subcritical
 and
 \begin{align}\label{LOGMOMENT}
  \int \limits_{ \{ \nu \in M_h(E)^{\circ} \ : \ \langle h,\nu \rangle > 1 \} } \log(\langle h, \nu \rangle ) H_2(d\nu) < \infty.
 \end{align}
 Then there exists a unique invariant measure $\pi \in \mathcal{P}(M_h(E))$. Moreover, for each $\mu \in \mathcal{P}(M_h(E))$,
 $P_t(\mu,\cdot) \longrightarrow \pi$ weakly as $t \to \infty$, 
 and 
\begin{align}\label{EQ:03}
  \mathcal{L}_{\pi}(f) = \exp\left( - \int \limits_{0}^{\infty}\psi(V_sf)ds \right),
  \qquad f \in B_h(E)^+.
 \end{align}
\end{Theorem}
A similar formula to \eqref{EQ:03} was obtained by Stannat \cite{S03a, S03b} for the case of a compact location space $E$ with $H_2 = 0$,
$\phi$ given by \eqref{EQ:19} and $h \equiv 1$.
Although the proof of Theorem \ref{TH:00} does not look very difficult, the particular case $E = \{1,\dots, d\}$ with general (nonlocal) branching and immigration mechanism was only recently established in \cite{JKR18} where affine processes on the canonical state space have been studied.

Without immigration (i.e. $\psi \equiv 0$) one has $\mathcal{L}_{\pi}(f) \equiv 1$ and hence $\pi = \delta_0$, where $0$ denotes the zero measure. This is not surprising as we work with subcritical superprocesses. Hence an introduction of immigration in not only possible, it is also necessary for the existence of a nontrivial invariant measure. A similar effect was recently observed in \cite{FK18} for birth-and-death processes in the continuum.

The rest of this section is devoted to the proof of Theorem \ref{TH:00}. Suppose that \eqref{LOGMOMENT} is satisfied.
Here and below we let $V_t$ be the unique solution to \eqref{EQ:00}.
Suppose that the corresponding superprocess is subcritical and let $\delta > 0$ be given by \eqref{EQ:07}.
We start with a simple but crucial stability estimate for solutions to 
\eqref{EQ:00}.
\begin{Proposition}
 There exists a constant $C > 0$ such that, for each $f \in B_h(E)^+$, one has
 \begin{align}\label{EXPONENTIAL:DECAY}
  V_tf(x) \leq C \| f \|_{h}h(x) e^{- \delta t}, \qquad t \geq 0, \ \ x \in E.
 \end{align}
\end{Proposition}
\begin{proof}
 Let $Q_t(\mu,d\nu)$ be the Markov kernel
 of a $(\xi, \phi, \psi = 0)$-superprocess given by \eqref{BRANCHING:LAPLACE}. 
 By Jensen inequality applied to the convex function $t \longmapsto e^{-t}$ we obtain
 \begin{align*}
  \exp\left( - \langle V_tf, \mu \rangle \right) 
  &= \int \limits_{M_h(E)} e^{- \langle f, \nu \rangle} Q_t(\mu,d\nu)
  \\ &\geq \exp\left(- \int \limits_{M_h(E)} \langle f, \nu \rangle Q_t(\mu,d\nu) \right)
  = \exp\left( - \langle R_tf, \mu \rangle \right),
 \end{align*}
 where the last equality follows from Lemma \ref{LEMMA:00}.
 From this we conclude that 
 $V_tf(x) \leq R_tf(x) \leq \| f\|_h C h(x)e^{- \delta t}$
 which proves the assertion.
\end{proof}
 Based on this stability estimate we prove the following.
\begin{Lemma}\label{LEMMA:05}
 There exists a constant $C > 0$ such that, for each $f \in B_h(E)^+$, one has
 \begin{align}\label{EQ:02}
  \psi(V_sf) \leq C \| f \|_{h} e^{- \delta s}, \qquad s \geq 0
 \end{align}
 and hence $\int_{0}^{\infty}\psi(V_sf)ds \leq C \delta^{-1}\| f\|_h < \infty$.
\end{Lemma}
\begin{proof}
 Let $f \in B_h(E)^+$, $s \geq 0$ and $\nu \in M_h(E)^{\circ}$.
 Denote by $C > 0$ a generic constant which may vary from line to line and is independent of $s,\nu,f$. Then we obtain
 \begin{align*}
  \left( 1 - e^{- \langle V_sf, \nu \rangle } \right)
  &\leq \min \left\{1, \langle V_s f, \nu \rangle \right\}
  \\ &\leq \1_{ \{ \langle h, \nu \rangle \leq 1\} } \langle V_sf, \nu \rangle + C\1_{ \{ \langle h, \nu \rangle > 1 \} } \log( 1 + \langle V_s f, \nu \rangle )
  \\ &=: I_1 + I_2.
 \end{align*}
 For the first term we obtain from  \eqref{EXPONENTIAL:DECAY}
 \[
  I_1 \leq \1_{ \{ \langle h, \nu \rangle \leq 1\} } C \| f \|_{\infty} e^{- \delta s} \langle h, \nu \rangle.
 \]
 For the second term we obtain
 \begin{align*}
  I_2 &\leq \1_{ \{ \langle h, \nu \rangle > 1\} }\log( 1 + C \| f \|_{\infty}e^{- \delta s}\langle h, \nu \rangle )
  \\ &\leq \1_{ \{ \langle h, \nu \rangle > 1\} }C\| f \|_{\infty}e^{- \delta s}\left( 1 + \log( 1 + \langle h, \nu \rangle ) \right),
 \end{align*}
 where we have used, for $x = C \| f \|_{\infty}e^{-\delta s}$ and $y = \langle h, \nu \rangle$, the elementary estimate
 \begin{align}
  \notag \log( 1 + xy) 
  &\leq C \log(1 + x) \log(1+y) + C \min\{ \log(1 + x), \log(1+y)\}
  \\ \label{EQ:16} &\leq C x\left( 1 + \log(1+y) \right),
 \end{align}
 see the appendix in \cite{FJR18} for a proof.
 Inserting this into the definition of $\psi$ gives
 \begin{align*}
  \psi(V_sf) &= \langle V_sf, \beta \rangle + \int \limits_{M_h(E)^{\circ}}\left( 1 - e^{-\langle V_sf, \nu \rangle}\right)H_2(d\nu)
  \\ &\leq C\| f\|_h e^{-\delta s}\langle h,\beta\rangle
  + C\| f\|_h e^{-\delta s}\int \limits_{\{ \langle h,\nu \rangle \leq 1\}} \langle h, \nu \rangle H_2(d\nu)
  \\ &\ \ \ + C\| f\|_h e^{-\delta s} \int \limits_{ \{ \langle h,\nu \rangle > 1 \}} \left( 1 + \log(1 + \langle h,\nu \rangle )\right) H_2(d\nu).
 \end{align*}
 In view of \eqref{LOGMOMENT} and \eqref{EQ:09} we find that the right-hand side is finite, i.e. \eqref{EQ:02} holds. This proves the assertion.
\end{proof}
We are now prepared to complete the proof.

\begin{proof}[Proof of Theorem \ref{TH:00}]
Using  \eqref{EXPONENTIAL:DECAY} combined with Lemma \ref{LEMMA:05}
shows that, for each $f \in B_h(E)^+$,
\begin{align}\label{EQ:12}
 \int \limits_{M_h(E)}e^{- \langle f, \nu \rangle} P_t(\mu, d\nu)
 \longrightarrow \exp\left( - \int \limits_{0}^{\infty}\psi(V_sf)ds \right), 
 \qquad t \to \infty.
\end{align}
Since $\int_{0}^{\infty}\psi(V_sf)ds \leq C \delta^{-1}\| f\|_h$, one finds that the right-hand side is continuous at $f = 0$ and therefore
the Laplace transform of a unique $\pi \in \mathcal{P}(M_h(E))$,
apply e.g. Proposition \ref{LEMMA:01}.
It remains to show that $\pi$ is the unique invariant measure.

In order to show that $\pi$ is invariant, 
we use a similar argument to \cite{JKR18}. 
Fix $f \in B_h(E)^+$ and $t \geq 0$.
Then using $V_sV_t = V_{s+t}$ for $s,t \geq 0$ yields
\begin{align*}
 &\ \int \limits_{M_h(E)} e^{- \langle f, \nu \rangle} \left( \int \limits_{M_h(E)} P_t(\mu,d\nu) \pi(d\mu) \right)
  \\ &= \int \limits_{M_h(E)} \exp\left( - \langle V_tf, \mu \rangle - \int \limits_{0}^{t}\psi(V_sf)ds \right) \pi(d\mu)
 \\ &= \exp\left( - \int \limits_{0}^{t}\psi(V_sf)ds \right) \exp\left( - \int \limits_{0}^{\infty}\psi(V_sV_tf)ds \right)
 \\ &= \exp\left( - \int \limits_{0}^{t}\psi(V_sf)ds \right) \exp\left( - \int \limits_{t}^{\infty}\psi(V_{s}f)ds \right)
 \\ &= \exp\left( - \int \limits_{0}^{\infty}\psi(V_sf)ds \right)
 = \mathcal{L}_{\pi}(f).
\end{align*}
Hence by uniqueness of Laplace transforms
(see Proposition \ref{LEMMA:01}) we conclude
that $\pi$ is invariant. 

Let us prove that $\pi$ is the unique invariant measure.
Let $\pi'$ be another invariant measure.
For each $t \geq 0$ we obtain from \eqref{EQ:17}
\begin{align*}
 \mathcal{L}_{\pi'}(f)
 = \int \limits_{M_h(E)}\left( \int \limits_{M_h(E)}e^{-\langle f, \nu \rangle} P_t(\mu,d\nu)\right) \pi'(d\mu).
\end{align*}
Taking the limit $t \to \infty$ and using \eqref{EQ:12} gives
\[
 \mathcal{L}_{\pi'}(f) = \exp\left( - \int \limits_{0}^{\infty}\psi(V_sf)ds\right) = \mathcal{L}_{\pi}(f).
\]
Since $f$ was arbitrary, we conclude that $\pi = \pi'$.
\end{proof}

\section{Ergodicity and convergence rate}
In this section we study the convergence rate for $P_t(\mu,\cdot) \longrightarrow \pi$ in two 
different distances on $\mathcal{P}(M_h(E))$.
We define the Laplace distance on $\mathcal{P}(M_h(E))$ by
\begin{align}
 d_{\mathcal{L}}(\rho,\widetilde{\rho}) = \sup \limits_{ f \in B_h(E)^+ } \frac{| \mathcal{L}_{\rho}(f) - \mathcal{L}_{\widetilde{\rho}}(f)|}{\| f \|_{\infty}}, \qquad \rho,\widetilde{\rho} \in \mathcal{P}(M_h(E)).
\end{align}
It is not difficult to see that $(\mathcal{P}(M_h(E)), d_{\mathcal{L}})$ is a complete metric space. Moreover, convergence with respect to $d_{\mathcal{L}}$ implies weak convergence. For $\rho \in \mathcal{P}(M_h(E))$ we let
\begin{align}\label{EQ:10}
 P_t^*\rho(d\nu) = \int \limits_{M_h(E)}P_t(\mu,d\nu)\rho(d\mu), \qquad t \geq 0.
\end{align}
It describes the distribution of the $(\xi,\phi,\psi)$-superprocess at time $t \geq 0$ when its law at initial time $t = 0$ is given by $\rho$.
Clearly $\pi \in \mathcal{P}(M_h(E))$ is an invariant measure if and only if $P_t^* \pi = \pi$ holds for all $t \geq 0$.
The following is our first main result of this section.
\begin{Theorem}\label{TH:01}
 Suppose that the $(\xi,\phi,\psi)$-superprocess is subcritical 
 and $H_2$ satisfies \eqref{LOGMOMENT}.
 Let $\pi$ be the unique invariant measure given by Theorem \ref{TH:00}. Then there exists a constant $C > 0$ such that,
 for each $\rho \in \mathcal{P}(M_h(E))$, 
 \begin{align*}
  d_{\mathcal{L}}( P_t^*\rho, \pi) \leq C e^{-\delta t} \left( 1 + \int \limits_{M_h(E)} \log(1 + \langle h,\mu \rangle) \rho(d\mu) \right), \qquad t \geq 0,
 \end{align*}
 where $\delta$ is given by \eqref{EQ:07}.
\end{Theorem}
Based on the construction of $(\xi, \phi, \psi)$-superprocesses, it is natural to measure the rate of convergence in the distance $d_{\mathcal{L}}$. Surprisingly enough, such metric was to our knowledge not used in the literature so far.

Next we study the convergence rate with respect to the Wasserstein-1-distance introduced below.
Denote by $\| \cdot\|_{TV,h}$ the weighted total variation distance 
\[
 \|\mu - \widetilde{\mu}\|_{TV,h} = \int \limits_{E}h(x)|\mu - \widetilde{\mu}|(dx)
\]
where $|\mu - \widetilde{\mu}| = (\mu - \widetilde{\mu})_+ + (\mu - \widetilde{\mu})_-$ and $(\mu - \widetilde{\mu})_{\pm}$ denote the Hahn-Jordan decomposition of the signed measure $\mu - \widetilde{\mu}$. Using the homeomorphism $M_h(E) \ni \mu \longmapsto h(x)\mu(dx) \in M_{\1}(E)$ combined with 
\cite[Corollary 1.12]{L11} one can show that the Borel-$\sigma$-algebra generated by the weighted total variation distance 
coincides with \eqref{EQ:22}. Hence $\mathcal{P}(M_h(E))$ remains the same although the topologies generated by \eqref{EQ:21} and the weighted total variation distance do not coincide.

Given $\rho, \widetilde{\rho} \in \mathcal{P}(M_h(E))$, we call $H$ a coupling of $(\rho, \widetilde{\rho})$ if
it is a Borel probability measure on $M_h(E) \times M_h(E)$ whose marginals are given by $\rho$ and $\widetilde{\rho}$. Denote by $\mathcal{H}(\rho,\widetilde{\rho})$ the collection of all couplings of $(\rho, \widetilde{\rho})$, respectively.
The Wasserstein-1-distance on $\mathcal{P}_1(M_h(E))$ is defined by
\[
 W_1(\rho, \widetilde{\rho}) = \inf \left\{ \int \limits_{M_h(E) \times M_h(E)} \|\mu - \widetilde{\mu} \|_{TV,h} H(d\mu,d\widetilde{\mu}) \ : \ H \in \mathcal{H}(\rho, \widetilde{\rho}) \right\},
\]
where 
\[
 \mathcal{P}_1(M_h(E)) = \left\{ \rho \in \mathcal{P}(M_h(E)) \ : \ \int \limits_{M_h(E)} \langle h, \mu \rangle \rho(d\mu) < \infty \right\}.
\]
By classical results on Wasserstein distances, see \cite{V09}, 
one finds that $(\mathcal{P}_1(M_h(E)), W_1)$ is a complete metric space. Convergence with respect to $W_1$ is strictly stronger then weak convergence on $\mathcal{P}_1(M_h(E))$.
The following is our second main result of this section.
\begin{Theorem}\label{TH:02}
 Suppose that the $(\xi,\phi,\psi)$-superprocess is subcritical and $H_2$ satisfies \eqref{EQ:06}.
 Then the unique invariant measure $\pi$ belongs to $\mathcal{P}_1(M_h(E))$ and satisfies
 \begin{align}\label{EQ:01}
  \int \limits_{M_h(E)} \mu(dx) \pi(d\mu) 
  = \int \limits_{0}^{\infty} R_s^* a(dx) ds < \infty.
 \end{align}
 Moreover, there exists a constant $C > 0$ such that,
 for each $\rho \in \mathcal{P}_1(M_h(E))$,
 \begin{align}\label{EQ:08}
  W_1(P_t^*\rho, \pi) \leq \left( \int \limits_{M_h(E)} \langle h,\mu \rangle \rho(d\mu) + \int \limits_{0}^{\infty}\langle R_sh, a \rangle ds \right)C e^{- \delta t}, \qquad t \geq 0,
 \end{align}
 where $\delta$ is given by \eqref{EQ:07}.
\end{Theorem}
In the particular case $E = \{1,\dots, d\}$ based on the use of Stochastic Differential Equations some similar (and essentially stronger) estimates have been recently obtained in \cite{FJR18}.
Unfortunately the techniques developed there do not directly apply in this case. Our proof is strongly based on the representation of the Laplace transforms \eqref{BRANCHING:LAPLACE} combined with some properties of the Wasserstein-1-distance.

\subsection{Proof of Theorem \ref{TH:01}}
Let $C > 0$ be a generic constant which may vary from line to line. Take $f \in B_h(E)^+$, $t \geq 0$ and $\mu \in \mathcal{P}(M_h(E))$. Then
\begin{align*}
 |L_{P_t(\mu, \cdot)}(f) - L_{\pi}(f)|
 &= \left| \exp\left( - \langle V_tf, \mu \rangle - \int \limits_0^t \psi(V_sf)ds \right) - \exp\left( - \int \limits_0^{\infty}\psi(V_sf)ds \right) \right|
 \\ &\leq \exp\left( - \int \limits_0^t \psi(V_sf)ds \right)\left| e^{- \langle V_tf, \mu \rangle} - 1\right| 
 \\ &\ \ \ + \left| \exp\left( - \int \limits_0^t \psi(V_sf)ds\right) - \exp\left( - \int \limits_{0}^{\infty}\psi(V_sf)ds \right) \right|
 \\ &= I_1 + I_2.
\end{align*}
Using $|e^{-a} - e^{-b}| \leq 1 \wedge |a-b|$, $a,b \geq 0$, 
and then $1 \wedge a \leq C \log(1 + a)$, gives
\[
 I_1 \leq 1 \wedge \langle V_tf, \mu \rangle 
 \leq C \log\left( 1 + \langle V_tf, \mu \rangle\right)
 \leq C \| f \|_h e^{- \delta t} \left( 1 + \log( 1 + \langle h, \mu \rangle) \right),
\]
where we have used \eqref{EQ:16}.
Similarly we obtain for the second term 
\begin{align*}
 I_2 \leq \int \limits_{t}^{\infty}\psi(V_sf)ds 
  \leq C \| f \|_h \int \limits_{t}^{\infty}e^{- \delta s}ds 
  \leq C \delta^{-1} \| f \|_h e^{- \delta t}.
\end{align*}
Combining both estimates gives
\[
 d_L(P_t(\mu, \cdot), \pi) \leq C e^{- \delta t}\left( 1 + \log( 1 + \langle h, \mu \rangle) \right).
\]
Let $\rho \in \mathcal{P}(M_h(E))$ and let $H$ be any coupling of $(\rho,\pi)$. Using Lemma \ref{LEMMA:04} gives
\begin{align*}
 d_L(P_t^*\rho, \pi) &= d_L(P_t^*\rho, P_t^*\pi) 
 \\ &\leq \int \limits_{M_h(E) \times M_h(E)} d_L(P_t(\mu,\cdot), P_t(\widetilde{\mu},\cdot)) H(d\mu,d\widetilde{\mu})
 \\ &\leq \int \limits_{M_h(E) \times M_h(E)}C e^{- \delta t}\left( 1 + \log( 1 + \langle h, \mu \rangle) \right) H(d\mu,d\widetilde{\mu})
 \\ &= C e^{-\delta t} \left( 1 + \int \limits_{M_h(E)} \log(1 + \langle h,\mu \rangle) \rho(d\mu) \right).
\end{align*}
This completes the proof of Theorem \ref{TH:01}.

\subsection{Proof of Theorem \ref{TH:02}}
Using first $P_t(\mu,\cdot) \longrightarrow \pi$
weakly as $t \to \infty$ combined with the Lemma of Fatou gives
\begin{align*}
 \int \limits_{M_h(E)}\langle f,\nu\rangle \pi(d\nu)
 &\leq \liminf \limits_{t\to \infty}\int \limits_{M_h(E)} \langle f,\nu\rangle P_t(\mu,d\nu)
 \\ &= \liminf\limits_{t\to \infty} \left( \langle R_t f, \mu \rangle + \int \limits_{0}^{t} \langle R_sf, a \rangle \right)
 = \int \limits_{0}^{\infty} \langle R_sf, a \rangle ds < \infty,
\end{align*}
where we have used Lemma \ref{LEMMA:00} combined with \eqref{EXPONENTIAL:DECAY}. This proves \eqref{EQ:01}.
Denote by $Q(\mu,d\nu)$ the transition probabilities of the 
$(\xi,\phi,\psi = 0)$-superprocess. 
From \eqref{BRANCHING:LAPLACE} we obtain $P_t(\mu, \cdot) = Q(\mu,\cdot) \ast P(0, \cdot)$. Using first the convexity of the Wasserstein distance 
(see \cite[Theorem 4.8]{V09}) and then Lemma \ref{LEMMA:02} yields
\begin{align*}
 W_1(P_t^*\rho, \pi) &= W_1(P_t^*\rho, P_t^*\pi)
 \\ &\leq \int \limits_{M_h(E)\times M_h(E)}W_1(P_t(\mu,\cdot),P_t(\widetilde{\mu},\cdot)) H(d\mu,d\widetilde{\mu})
 \\ &\leq \int \limits_{M_h(E)\times M_h(E)}W_1(Q_t(\mu,\cdot),Q_t(\widetilde{\mu},\cdot)) H(d\mu,d\widetilde{\mu}).
\end{align*}
Estimating the Wasserstein distance from above by the particular choice of coupling $Q_t(\mu, d\nu)Q_t(\widetilde{\mu},d\widetilde{\nu})$ and then using 
$\| \nu - \widetilde{\nu} \|_{TV,h} \leq \| \nu \|_{TV,h} + \| \widetilde{\nu} \|_{TV,h} = \langle h, \nu\rangle + \langle h, \widetilde{\nu} \rangle$ yields
\begin{align*}
 W_1(Q_t(\mu,\cdot),Q_t(\widetilde{\mu},\cdot))
 &\leq \int \limits_{M_h(E)\times M_h(E)} \| \nu - \widetilde{\nu}\|_{TV,h} Q_t(\mu,d\nu)Q_t(\widetilde{\mu}, d\widetilde{\nu})
 \\ &\leq \int \limits_{M_h(E)} \langle h, \nu \rangle Q_t(\mu,d\nu) + \int \limits_{M_h(E)} \langle h, \nu \rangle Q_t(\widetilde{\mu}, d\nu)
 \\ &= \langle R_th, \mu \rangle + \langle R_th,\widetilde{\mu}\rangle
 \\ &\leq C e^{- \delta t}\left( \langle h, \mu \rangle + \langle h, \widetilde{\mu}\rangle \right),
\end{align*}
where the last equality follows from Lemma \ref{LEMMA:00} with $f = h$.
Hence we obtain
\begin{align*}
 W_1(P_t^*\rho,\pi) &\leq C e^{-\delta t} \int \limits_{M_h(E)\times M_h(E)}\left( \langle h, \mu \rangle + \langle h, \widetilde{\mu}\rangle \right)H(d\mu,d\widetilde{\mu})
 \\ &= \left( \int \limits_{M_h(E)} \langle h, \mu \rangle \rho(d\mu) + \int \limits_{M_h(E)} \langle h, \mu \rangle \pi(d\mu)\right)C e^{-\delta t}.
\end{align*}
In view of \eqref{EQ:01}, the assertion is proved.

\section{Examples}

\subsection{Super-L\'evy processes}
Let $E = \R^d$ and denote by $C_0(\R^d)$ the Banach space of continuous functions vanishing at infinity. Let $\xi$ be a L\'evy process with generator $(A,D(A))$ on $C_0(\R^d)$ and transition semigroup
\[
 p_t^{\xi}f(x) = \E_x[ f(\xi_t) ], \qquad f \in C_0(\R^d).
\]
Let $\phi$ be the branching mechanism given by 
\[
 \phi(x,f) = b f(x) + c f(x)^2 + \int \limits_0^{\infty}\left( e^{-zf(x)} - 1 + f(x)z \right)m(x,dz),
\]
where $b \in \R$, $c \geq 0$ and $\sup_{x \in \R^d}\int_{0}^{\infty} z \wedge z^2 n(x,dz) < \infty$.
For vanishing immigration (i.e. $\psi \equiv 0$) 
and in dimension $d = 1$ the corresponding branching process was constructed in \cite{HLY14}. Below we introduce immigration of mass into the system. 
Let $\rho$ be a Borel probability measure on $\R^d$ and consider the immigration mechanism
\[
 \psi(f) = \int \limits_{0}^{\infty}\left( 1 - e^{- z \langle f, \rho\rangle} \right)G(dz),
\]
where $G$ is a Borel measure on $(0,\infty)$ such that $\int_0^{\infty}zG(dz) < \infty$.

This model describes a system of particles as follows. New mass immigrates according to a Poisson random measure with distribution $\rho$ and intensity described by the measure $G$. After mass is placed inside $\mathrm{supp}(\rho) \subset \R^d$, it evolves according to the prescribed L\'evy process $\xi$ performing randomly some branching described by the mechanism $\phi$. Below we briefly explain how the general results of this work can be applied to this case.

Suppose that there exists a strictly positive, continuous function $h \in C_0(\R^d)$ which belongs to the domain of the extended generator and satisfies $Ah \leq K h$ for some constant $K > 0$. Then conditions (A1) -- (A3) are satisfied. 
Since $\phi$ is local, we have $\gamma(x,dy) = 0$ and hence $r_t(x,dy) = e^{-bt}p_t^{\xi}(x,dy)$
which yields
\[
 \int \limits_{E}h(y)r_t(x,dy) = e^{-bt} \int \limits_{E}h(y)p_t^{\xi}(x,dy) \leq e^{(K - b)t}h(x).
\]
This shows that the super-L\'evy process is subcritical
provided that $b > K$.
In particular, under the suitable integrability condition
\[
 \int \limits_{1}^{\infty}\log(z)G(dz) < \infty,
\]
such process has a unique invariant measure, see Section 3.

\subsection{Infinite-type continuous-state branching processes with immigration}
Let $E$ be a countable (finite or infinite) discrete set equipped with the discrete topology. The reader may think about the particular cases $E = \Z^d$ and $\E = \N$.
Let $h = (h(x))_{x \in E}$ 
be a strictly positive sequence of elements indexed by $E$. 
Define the weighted space of summable sequences
\[
 \ell_h^1 = \left\{ \mu = (\mu(x))_{x \in E} \ | \ \sum \limits_{x \in E}\mu(x) h(x) < \infty \right\}
\]
and the weighted space of bounded sequences
\[
 \ell_h^{\infty} = \left\{ f = (f(x))_{x \in E} \ | \ \sup \limits_{x \in E}\frac{|f(x)|}{h(x)} < \infty \right\}.
\]
The corresponding pairing is given by
\[
 \langle f, \mu \rangle = \sum \limits_{x \in E}f(x) \mu(x), 
 \qquad f \in \ell_h^{\infty}, \ \ \mu \in \ell_h^{1}.
\]
Below we consider a particle system with sites indexed by $E$ and noncompact spins $\R_+$, i.e. the state space is $\R_+^E$. 
Such particle system is supposed to be described by the following objects:
\begin{enumerate}[leftmargin = *]
 \item[(i)] Let $\xi$ be a conservative Markov chain on $E$ whose generator is given by 
 \[
  (Af)(x) = \sum \limits_{y \in E}(f(y) - f(x))q(x,y), \qquad f \in B_h(E),
 \]
 where $q(x,y) \geq 0$ is such that 
 \begin{align}\label{EQ:28}
  \sup \limits_{x \in E}\sum \limits_{y \in E}q(x,y) < \infty.
 \end{align}
 Moreover there exists a constant $C > 0$ satisfying
 \begin{align}\label{EQ:29}
  \sum \limits_{y \in E}h(y) q(x,y) \leq C h(x) + h(x)\sum \limits_{y \in E}q(x,y), \qquad x \in E.
 \end{align}
 \item[(ii)] $(b(x))_{x \in E}$ a bounded sequence.
 \item[(iii)] $(c(x))_{x \in E}$ is a nonnegative sequence such that $(h(x)c(x))_{x \in E}$ is bounded.
 \item[(iv)] $(\eta(x,y))_{x,y \in E}$
 is an infinite matrix with nonnegative entries such that
 \[
  \sum \limits_{y \in E}h(y) \eta(x,y) \leq C h(x), \qquad x \in E
 \]
 for some constant $C > 0$.
 \item[(v)] $H_1(x, d\eta)$ is a collection of $\sigma$-finite measures on $\ell_h^{1} \backslash \{0\}$ indexed by $E$ and satisfying
 \[
 \int \limits_{\ell_h^1 \backslash \{0\}}\left( \langle h,\nu \rangle \wedge \langle h, \nu \rangle^2 + \langle h, \nu_x \rangle \right) H_1(x,d\nu) \leq C h(x), \qquad x \in E,
 \]
 for some constant $C > 0$.
 \item[(vi)] $(\beta(x))_{x \in E} \in \ell_h^1$ is a nonnegative sequence.
 \item[(vii)] $H_2$ is a Borel measure on $\ell_h^1 \backslash \{0\}$ satisfying
 \begin{align*}
  \int \limits_{\ell_h^1 \backslash \{0\}} 1 \wedge \langle h, \nu \rangle H_2(d\nu) < \infty.
 \end{align*}
\end{enumerate}
For $x \in E$ and $(f(x))_{x \in E} \in \ell_h^{\infty}$ nonegative, we introduce the branching mechanism
 \begin{align*}
 \phi(x,f) &= c(x)f(x)^2 + b(x)f(x) - \sum \limits_{y \in E}f(y) \eta(x,y)
 \\ &\ \ \ + \int \limits_{\ell_h^1 \backslash \{0\}}\left( e^{- \langle f,\nu\rangle} - 1 + f(x) \nu(x) \right) H_1(x,d\nu)
\end{align*}
and immigration mechanism
\begin{align}\label{EQ:26}
 \psi(f) = \sum \limits_{x \in E}f(x)\beta(x) + \int \limits_{\ell_h^1 \backslash \{0\}}\left( 1 - e^{- \langle f, \nu \rangle}\right)H_2(d\nu).
\end{align}
\begin{Theorem}
 Suppose that conditions (i) -- (vii) are satisfied. Then
 \begin{enumerate}[leftmargin = *]
  \item[(a)] For each nonnegative sequence $f \in \ell_h^{\infty}$ there exists a unique locally bounded solution $(v_t(x))_{x \in E}$ in $\ell_h^{\infty}$ to 
  \begin{align}\label{EQ:24}
   \frac{d v_t(x)}{d t} = Av_t(x) - \phi(x,v_t), \qquad v_0(x) = f(x), \ \ x \in E.
  \end{align}
  \item[(b)] There exists a unique Markov kernel $P_t(\eta,d\mu)$ whose Laplace transform is given by
  \[
   \int\limits_{\ell_h^1}e^{- \langle f, \mu \rangle} P_t(\eta, d\mu) = \exp \left( - \langle V_tf, \eta \rangle - \int \limits_0^t \psi(V_sf) ds \right).
  \]
  where $f \in \ell_h^{\infty}$ is nonnegative and $V_tf(x) = v_t(x)$.
  \item[(c)] Assume that there exists a constant $\delta > 0$ satisfying
  \begin{align}\label{EQ:25}
   \sum \limits_{y \in E}h(y) \left( q(x,y) + \gamma(x,y)\right) \leq \left( b(x) + \sum \limits_{y \in E}q(x,y) \right)h(x) - \delta h(x), 
  \end{align}
  for each $x \in E$ where 
  \[
   \gamma(x,y) = \eta(x,y) + \int \limits_{\ell_h^1 \backslash \{0\}} \1_{ \{ x \neq y\} } \nu(y) H_1(x,d\nu).
  \]
  Then the corresponding superprocess is subcritical.
 \end{enumerate}
\end{Theorem}
\begin{proof}
 Note that $M_h(E)$ may be identified by $\mu \longmapsto (\mu(\{x\}))_{x \in E} \equiv (\mu(x))_{x \in E}$ with $\ell_h^1$.
 Similarly we can identify $B_h(E)$ by $f \longmapsto (f(x))_{x \in E}$ with $\ell_h^{\infty}$. 
 Condition \eqref{EQ:28} implies that $A$ is a bounded operator on $B_{\1}(E)$ and by \eqref{EQ:29} it is also bounded on $B_h(E)$. Hence condition (A1) is satisfied. Conditions (A2) and (A3) are immediate consequences of (ii) -- (vii) so that
 assertions (a) and (b) follow from Theorem \ref{TH:00}. Let us prove assertion (c). Observe that the right-hand side in \eqref{EQ:24} satisfies
 \begin{align*}
  Av_t(x) - \phi(x,v_t) = \widetilde{\phi}(x,v_t),
 \end{align*}
 where $\widetilde{\phi}(x,f)$ is a new branching mechanism 
 given by 
 \begin{align*}
  \widetilde{\phi}(x,f) &= \widetilde{b}(x)f(x) + c(x)f(x)^2 
  - \sum \limits_{y \in E}f(y) \widetilde{\gamma}(x,y)
  \\ &\ \ \ + \int \limits_{\ell_h^1 \backslash \{0\}}\left( e^{- \langle f,\nu \rangle} - 1 + \langle f, \nu \rangle \right) H_2(x,d\nu)
 \end{align*}
  with 
 \begin{align*}
  \widetilde{b}(x) = b(x) + \sum \limits_{y \in E}q(x,y), \qquad 
  \widetilde{\gamma}(x,y) = q(x,y) + \gamma(x,y).
 \end{align*}
 Hence $(v_t)_{t \geq 0}$ given by \eqref{EQ:24} is also the unique solution to
 \[
  \frac{dv_t(x)}{dt} = - \widetilde{\phi}(x,v_t), \qquad v_0(x) = f(x).
 \]
 Using this new representation (which has no spatial motion) we first observe that it still satisfies conditions (A1) -- (A3) for the particular choice $\xi_t = \xi_0$, $t \geq 0$.
 We proceed to compute the semigroup $(R_t)_{t \geq 0}$.
 Since $\xi_t = \xi_0$, we have 
 $p_t^{\xi}f(x) = f(x)$ and hence in view of \eqref{EQ:34}
 the semigroup $(R_t)_{t \geq 0}$ is the unique solution to 
 \[
  R_th(x) = h(x) + \int \limits_{0}^{t}R_sBh(x) ds,
 \]
 i.e. $R_t = e^{tB}$,
 where $B$ is a bounded linear operator on $B_h(E)$ given by
 \[
  Bf(x) = \sum \limits_{y \in E}f(y) \widetilde{\gamma}(x,y) - \widetilde{b}(x)f(x).
 \]
 By \eqref{EQ:25} one has $Bh(x) \leq - \delta h(x)$ which shows that $R_th(x) \leq e^{- \delta t}h(x)$.
 Since 
 \[
  Bf(x) = Af(x) + \sum \limits_{y \in E}f(y) \gamma(x,y) - b(x)f(x),
 \]
 we conclude the assertion.
\end{proof}
Recently in \cite{KP18} the authors have studied local and global extincition for the particular choice 
$E = \N$, $h(x) \equiv 1$ with branching mechanism 
\begin{align}\label{EQ:27}
 \phi(x,f) &= b(x)f(x) + c(x)f(x)^2 + \int \limits_{0}^{\infty}\left( e^{- zf(x)} - 1 + zf(x) \right)G_0(x,dz)
 \\ \notag &\ \ \ - g(x)\left( d(x)\langle f, \pi_x\rangle + \int \limits_{0}^{\infty}\left( 1 - e^{- z \langle f, \pi_x \rangle} \right) G_1(x,dz) \right),
\end{align}
immigration mechanism $\psi \equiv 0$
and trivial spatial motion, i.e. $q(x,y) \equiv 0$.
Here $c(x),g(x),d(x) \geq 0$ and $b(x) \in \R$ are bounded sequences and $(\pi_x)_{x \in \N}$ is a family of probability distributions on $\N$ with $\pi_x(\{x\}) = 0$ for all $x \in \N$, and $z \wedge z^2 G_0(x,dz)$ and $z G_1(x,dz)$ are bounded kernels from $\N$ to $(0,\infty)$.
Note that the first part in $\phi$ corresponds to local branching while the second part is nonlocal branching which couples different lattice points in a nontrivial way.
Below we introduce immigration to this model and show that it satisfies conditions (i) -- (vii).
\begin{Corollary}
 Let $E = \N$, $h(x) \equiv 1$, $\phi$ be given by \eqref{EQ:27} and $\psi$ by \eqref{EQ:26}. 
 Then conditions (i) -- (vii) are satisfied.
 If there exists $\delta > 0$ such that
 \[
  g(x)\left( d(x) + \int \limits_{0}^{\infty}z G_1(x,dz)\right) \leq b(x) - \delta, \qquad x \in E,
 \]
 then \eqref{EQ:25} is satisfied and the corresponding superprocess is subcritical.
\end{Corollary}
\begin{proof}
 Let $\eta(x,dy) = g(x) d(x) \pi_x(dy)$ and set 
  \[
  H_1(x,d\nu) = \int \limits_0^{\infty}\delta_{z \delta_x}(d\nu) G_0(x,dz) + \int \limits_{0}^{\infty}\delta_{z \pi_x}(d\nu) g(x) G_1(x,dz),
 \]
 then it is easily seen that conditions (i) -- (vii) are satisfied. The second assertion follows by direct computation.
\end{proof}

\appendix
\section*{Appendix}
\renewcommand{\thesection}{A} 

\subsection{Proof of Theorem \ref{TH:00}}
\begin{proof}
Assertion (a) is a particular case of \cite[Theorem 6.3]{L11}.
For later use we sketch the most important step in the proof.
Define a new transition semigroup $\widetilde{p}_t^{\xi}(x,dy) = e^{-\alpha t} \frac{h(y)}{h(x)}p_t^{\xi}(x,dy)$ and branching mechanism $\widetilde{\phi}(x,f) = h(x)^{-1}\phi(x,hf) - \alpha f(x)$. Then it was shown that there exists a unique locally bounded solution $u_t(x,f)$ to
\[
 u_t(x,f) = \widetilde{p}_t^{\xi}f(x) - \int \limits_{0}^{t} \int \limits_{E} \widetilde{\phi}(y, u_{t-s}) \widetilde{p}_{s}^{\xi}(x,dy) ds, \qquad f \in B_{\1}(E)^+.
\]
This solution defines a comulant semigroup $U_tf(x) = u_t(x,f)$, i.e.
\begin{align}\label{EQ:31}
 U_tf(x) = \int \limits_{E}f(y) \widetilde{\lambda}_t(x,dy) 
 + \int \limits_{M(E)^{\circ}} \left( 1 - e^{- \langle f, \nu \rangle } \right) \widetilde{L}_t(x,d\nu),
\end{align}
where $\widetilde{\lambda}_t$ and $\widetilde{L}_t$ are bounded kernels for each $t \geq 0$. Finally it was shown that 
\begin{align}\label{EQ:32}
 V_tf(x) := h(x) U_t(h^{-1}f)(x), \qquad f \in B_h(E)^+
\end{align}
is the desired solution to \eqref{EQ:00}, i.e. assertion (a) is proved.
By uniqueness we find that $(V_t)_{t \geq 0}$ is a nonlinear semigroup.
Invoking \eqref{EQ:31} and \eqref{EQ:32} we find that
\[
 \lambda_t(x,dy) = \frac{h(x)}{h(y)}\widetilde{\lambda}_t(x,dy), \qquad 
 L_t(x,d\nu) = h(x)\widetilde{L}(x, \cdot) \circ I^{-1}(d\nu),
\]
where $I: M_{\1}(E)^{\circ} \longrightarrow M_{h}(E)^{\circ}$, $I(\nu) = h^{-1}\nu$.
This proves assertion (b).
Assertion (c) can be obtained as follows. 
Define $H_2^{(n)}(d\nu) = \1_{ \{ \langle h, \nu \rangle > \frac{1}{n} \} }H_2(d\nu)$ and let $\psi_n$ be given by \eqref{EQ:13} with $H_2$ replaced by $H_2^{(n)}$. 
By \cite[Proposition 9.17]{L11} there exists a unique Markov kernel
$P^{(n)}_t(\mu,d\nu)$ satisfying \eqref{BRANCHING:LAPLACE} with $\psi$ replaced by $\psi_n$. 
Taking the limit $n \to \infty$ and using Proposition \ref{LEMMA:01} proves the existence of $P_t(\mu,d\nu)$ satisfying \eqref{BRANCHING:LAPLACE}. Since $V_{s+t} = V_s V_t$,
it is not difficult to see that $P_t(\mu,d\nu)$ is a Markov kernel.
\end{proof}

\subsection{Proof of Lemma \ref{LEMMA:06}}
Let $(U_t)_{t \geq 0}$ be the comulant semigroup given by \eqref{EQ:31} and let $(\widetilde{R}_t)_{t \geq 0}$ be the semigroup on $B_{\1}(E)$ obtained from
\[
 \widetilde{R}_tf(x) = \int \limits_{E}f(y) \widetilde{p}_t^{\xi}(x,dy)
 + \int \limits_{0}^{t} \int \limits_{E} \left( \widetilde{\Gamma}\widetilde{R}_{t-s}f(y) - \widetilde{b}(y)\widetilde{R}_{t-s}f(y) \right)\widetilde{p}^{\xi}_s(x,dy)ds,
\]
where $\widetilde{p}_t^{\xi}(x,dy) = e^{-\alpha t} \frac{h(y)}{h(x)}p_t^{\xi}(x,dy)$, $\widetilde{b}(x) = b(x) - \alpha$ and
\begin{align*}
 \widetilde{\Gamma}f(x) = \int \limits_{E}f(y)\widetilde{\gamma}(x,dy),
 \ \ \widetilde{\gamma}(x,dy) = \frac{h(y)}{h(x)}\gamma(x,dy).
\end{align*}
It was shown in \cite[Proposition 2.18 and 2.24]{L11} that 
$U_tf \leq \widetilde{R}_tf$ and
$\frac{1}{\e}U_t(\e f)(x) \nearrow \widetilde{R}_tf(x)$ as $\e \searrow 0$ for each $f \in \B_{\1}(E)^+$. 

Observe that $h(x) \widetilde{R}_t(h^{-1}f)(x)$ defines a semigroup of bounded linear operators on $B_h(E)^+$ which satisfies \eqref{EQ:34}. Since the operator $Bf(x) = \Gamma f(x) - b(x)f(x)$ is a bounded linear operator on $B_h(E)$, it follows that \eqref{EQ:34} has a unique solution, i.e. we get $R_tf(x) = h(x)\widetilde{R}_t(h^{-1}f)(x)$. 
Using representation \eqref{EQ:30} we easily find that
\[
 \frac{1}{\e}V_t(\e f)(x) = h(x) \frac{1}{\e}U_t( \e h^{-1}f)(x) \nearrow h(x) \widetilde{R}_t( h^{-1}f)(x) = R_tf(x).
\]
Hence using \eqref{EQ:32} and this limit one finds  \eqref{EQ:33}.

\subsection{Proof of Proposition \ref{LEMMA:00}}
 Fix $\mu \in M_h(E)$ and let $f \in B_h(E)^+$.
 By linearity and definition of the adjoint semigroup $(R_t^*)_{t \geq 0}$ it suffices to prove for each $t \geq 0$
 \[
  \int \limits_{M_h(E)}\langle f, \nu\rangle P_t(\mu,d\nu)
  = \langle R_t f, \mu \rangle + \int \limits_{0}^{t} \left(\langle R_sf, \beta \rangle + \int \limits_{M_h(E)^{\circ}}\langle R_sf, \nu \rangle H_2(d\nu) \right) ds.
 \]
We start with one auxilliary result.
\begin{Lemma}
 For each $f \in B_h(E)^+$ and $t \geq 0$ one has
 \[
  \lim \limits_{\e \searrow 0} \frac{1}{\e}\int \limits_0^{t} \psi(V_s(\e f)) ds 
  = \int \limits_{0}^{t} \left(\langle R_sf, \beta \rangle + \int \limits_{M_h(E)^{\circ}}\langle R_sf, \nu \rangle H_2(d\nu) \right) ds.
\]
\end{Lemma}
\begin{proof}
 Write
 \begin{align*}
  \frac{1}{\e} \int \limits_0^t \psi(V_s(\e f)) ds
  &= \int \limits_0^t \int \limits_{E} \frac{V_s(\e f)(x)}{\e} \beta(dx) ds + \int \limits_0^t \int \limits_{M_h(E)^{\circ}} \frac{1 - e^{- \langle V_s(\e f), \nu \rangle}}{\e} H_2(d\nu) ds.
 \end{align*}
 Next using $\frac{V_s(\e f)}{\e} \nearrow R_tf$ pointwise
 and 
 \[
  \lim \limits_{\e \searrow 0}\frac{1 - e^{- \langle V_s(\e f), \nu \rangle}}{\e}
  = \langle R_tf, \nu \rangle, \qquad \nu \in M_h(E)^{\circ}
 \]
 with 
  \[
  \frac{1 - e^{- \langle V_s(\e f), \nu \rangle}}{\e} \leq \frac{1}{\e} \langle V_s(\e f), \nu \rangle \leq \langle R_tf, \nu \rangle
 \]
 the assertion follows by dominated convergence.
\end{proof}
Based on this observation we complete the proof of Proposition \ref{LEMMA:00}.
\begin{proof}[Proof of Proposition \ref{LEMMA:00}]
 Fix $f \in B_h(E)^+$ and take $\e > 0$. 
 Applying first monotone convergence and then \eqref{BRANCHING:LAPLACE} gives
 \begin{align*}
  \int \limits_{M_h(E)} \langle f, \nu \rangle P_t(\mu, d\nu)
  &= \lim \limits_{\e \searrow 0}\int \limits_{M_h(E)} \frac{1 - e^{- \langle \e f, \nu \rangle}}{\e} P_t(\mu, d\nu) 
  \\ &= \lim \limits_{\e \searrow 0}\frac{1 - \exp\left( - \langle V_t(\e f), \mu \rangle - \int \limits_{0}^{t} \psi( V_s(\e f)) ds \right)}{\e}.
  \\ &= \lim \limits_{\e \searrow 0}\left( \int \limits_{E} \frac{V_t(\e f)(x)}{\e}\mu(dx) + \frac{1}{\e}\int \limits_{0}^{t}\psi(V_s(\e f)) ds \right)
  \\ &= \langle R_tf, \mu \rangle +
  \int \limits_{0}^{t} \left(\langle R_sf, \beta \rangle + \int \limits_{M_h(E)^{\circ}}\langle R_sf, \nu \rangle H_2(d\nu) \right) ds < \infty.
 \end{align*}
 This proves the assertion.
\end{proof}

\subsection{Some results of distances on $\mathcal{P}(M_h(E))$}

\begin{Lemma}\label{LEMMA:04}
 Let $\rho, \widetilde{\rho} \in \mathcal{P}(M_h(E))$ and $H \in \mathcal{H}(\rho, \widetilde{\rho})$. Then
 \begin{align*}
  d_{\mathcal{L}}(P_t^*\rho, P_t^*\widetilde{\rho})
  \leq \int \limits_{M_h(E) \times M_h(E)} d_{\mathcal{L}}(P_t(\mu,\cdot), P_t(\widetilde{\mu},\cdot))H(d\mu, d\widetilde{\mu}).
 \end{align*}
\end{Lemma}
\begin{proof}
 Take $f \in B_h(E)^+$. Using the definition of $P_t^*$ and the fact that $H$ is a coupling of $(\rho, \widetilde{\rho})$ gives
 \begin{align*}
  |\mathcal{L}_{P_t^*\rho}(f) - \mathcal{L}_{P_t^*\widetilde{\rho}}(f)|
  &= \left| \int \limits_{M_h(E)} \int \limits_{M_h(E) \times M_h(E)} e^{- \langle f,\nu \rangle} \left( P_t(\mu, d\nu) - P_t(\widetilde{\mu}, d\nu) \right)H(d\mu, d\widetilde{\mu}) \right|
  \\ &\leq \int \limits_{M_h(E) \times M_h(E)} \left| \mathcal{L}_{P_t(\mu,\cdot)}(f) - \mathcal{L}_{P_t(\widetilde{\mu},\cdot)}(f)\right| H(d\mu,d\widetilde{\mu})
  \\ &\leq \| f \|_h \int \limits_{M_h(E) \times M_h(E)} d_{\mathcal{L}}(P_t(\mu,\cdot), P_t(\widetilde{\mu}, \cdot)) H(d\mu, d\widetilde{\mu}).
 \end{align*}
 Since $f$ was arbitrary, the assertion is proved.
\end{proof}

\begin{Lemma}\label{LEMMA:02}
Let $\rho, \widetilde{\rho}, \thinspace g \in\mathcal{P}(M_h(E))$. 
Then $W_{1}(\rho \ast g,\widetilde{\rho}\ast g)\leq W_{1}( \rho,\widetilde{\rho})$.
\end{Lemma}
\begin{proof}
 For $F: M_h(E) \longrightarrow \R$ define $\Vert F \Vert_{\mathrm{Lip}} = \sup_{\mu \neq \widetilde{\mu}} \frac{|F(\mu) - F(\widetilde{\mu})|}{\| \mu - \widetilde{\mu}\|_{TV,h}}$. Using the Kantorovich-Duality we obtain
\begin{align*}
W_{1}(\rho \ast g, \widetilde{\rho}\ast g) 
&= \sup_{\| F \|_{\mathrm{Lip}} \leq 1} \left| \int \limits_{M_h(E)} F(\mu) (\rho\ast g)(d\mu) - \int \limits_{M_h(E)} F(\mu)(\widetilde{\rho}\ast g)(d\mu)\right| 
\\ &= \sup_{\| F \|_{\mathrm{Lip}} \leq 1} \left| \int \limits_{M_h(E)} F_g(\mu) \rho(d\mu) - \int \limits_{M_h(E)} F_g(x)\widetilde{\rho}(d\mu)\right| 
\\ &\leq \sup_{\| F \|_{\mathrm{Lip}} \leq 1} \left| \int \limits_{M_h(E)} F(\mu) \rho(d\mu) - \int \limits_{M_h(E)} F(\mu) \widetilde{\rho}(d\mu)\right| = W_1(\rho, \widetilde{\rho}),
\end{align*}
where we used that $F_g(\mu)=\int_{ M_h(E)} F( \mu + \widetilde{\mu})g(d\widetilde{\mu})$ satisfies $\| F_g\|_{\mathrm{Lip}} \leq 1$.
\end{proof}

\bibliographystyle{amsplain}
\addcontentsline{toc}{section}{\refname}\bibliography{Bibliography}

\end{document}